\documentclass[oneside, 12pt]{article}
\usepackage[margin=3cm]{geometry}

\usepackage{amsmath}
\usepackage{amssymb}
\usepackage{enumitem}	
\usepackage{amsthm}
\usepackage{hyperref}
\hypersetup{
  colorlinks=true,
  citecolor=black,
  linkcolor=black,
  urlcolor=black,
  breaklinks=true,
  bookmarksnumbered=true,
  bookmarksopen=true,
  pdftitle={The Worm Calculus}
  pdfauthor={Ana de Almeida Borges and Joost J. Joosten},
  pdfsubject={},
  pdfcreator={Ana de Almeida Borges},
  pdfkeywords={}
}
\usepackage[numbers]{natbib}   

\newcounter{dummy}
\makeatletter
\newcommand\myitem[1][]{\item\refstepcounter{dummy}\def\@currentlabel{#1}}
\makeatother

\theoremstyle{definition}
\newtheorem{theorem}{Theorem}[section]
\newtheorem{definition}[theorem]{Definition}
\newtheorem{lemma}[theorem]{Lemma}
\newtheorem{corollary}[theorem]{Corollary}
\newtheorem{proposition}[theorem]{Proposition}

\newcommand{\la}{\langle}
\newcommand{\ra}{\rangle}

\newcommand{\glp}{{\mathsf{GLP}}}
\newcommand{\gl}{\mathsf{GL}}
\newcommand{\Worms}{{\mathbb W}}

\newcommand{\RC}{{\sf RC}}
\newcommand{\RCzero}{{{\sf RC}^0}}

\newcommand{\WC}{{\sf WC}}

\newcommand{\vdashr}{\vdash_{\RC}}
\newcommand{\vdashw}{\vdash_{\WC}}

\newcommand{\len}[1]{|#1|}

\title{The Worm Calculus}
\author{Ana de Almeida Borges\thanks{\url{ana.agvb@gmail.com}} \and Joost J. Joosten\thanks{\url{jjoosten@ub.edu}}}
\date{August 2018}

\begin{document}

\maketitle

  \begin{abstract}
	We present a propositional modal logic $\WC$, which includes a logical \emph{verum} constant $\top$ but does not have any propositional variables. Furthermore, the only connectives in the language of $\WC$ are consistency-operators $\la \alpha \ra$ for each ordinal $\alpha$. As such, we end up with a class-size logic. However, for all practical purposes, we can consider restrictions of $\WC$ up to a given ordinal. Given the restrictive signature of the language, the only formulas are iterated consistency statements, which are called worms. The theorems of $\WC$ are all of the form $A \vdash B$ for worms $A$ and $B$.
	The main result of the paper says that the well-known strictly positive logic $\RC$, called Reflection Calculus, is a conservative extension of $\WC$. As such, our result is important since it is the ultimate step in stripping spurious complexity off the polymodal provability logic $\sf{GLP}$, as far as applications to ordinal analyses are concerned. Indeed, it may come as a surprise that a logic as weak as $\WC$ serves the purpose of computing something as technically involved as the proof theoretical ordinals of formal mathematical theories. 
  \end{abstract}

  \textbf{Keywords:}
	Provability logic,
	strictly positive logics,
	closed fragment,
	feasible fragments,
	Reflection Calculus,
	ordinal notations.

\section{Introduction}

Quite some interest has arisen in feasible fragments of modal logics recently. One of the common goals is to find fragments with good computational properties that still maintain a decent amount of expressibility. Description logics and their applications to database theory \cite{Artale2009} are a good example of this. 

The current paper also studies fragments of modal logic, but coming from a different tradition. Our starting point is $\glp$: a polymodal version of G\"odel-L\"ob's provability logic as introduced by Japaridze \cite{Japaridze1988}. The logic $\glp$ is a propositional modal logic which in its simplest version has a modality for each natural number. Although this logic is known to be $\sf PSPACE$-complete \cite{Shapirovsky2008}, it behaves rather ghastly. While complete with respect to topological semantics \cite{BeklemishevGabelaia2013}, $\glp$ is easily seen to be frame-incomplete.

The logic $\glp$ has received a substantial amount of interest due to its applications to ordinal analysis \cite{Beklemishev2004}. The variable-free fragment $\glp^0$ of $\glp$ actually suffices for various purposes. Going from $\glp$ to $\glp^0$ is then a first weakening leading up to our final system $\WC$ to be introduced below. 

The reason why $\glp^0$ is still suitably expressible lies in the fact that terms in it can be read in various ways. One can conceive of these terms as consistency statements or reflection principles. Furthermore, natural fragments of arithmetic are denoted by terms. The simplest terms of $\glp^0$ are iterated consistency statements, and they are called \emph{worms} due to their relation to the heroic worm battle \cite{Beklemishev2006_WormPrinciple}.
The worms modulo provable equivalence can be ordered, so that they can also be conceived of as ordinals \cite{FernandezDuqueJJJ2014}. Apart from their interpretation as consistency statements, reflection principles, fragments of arithmetic, or ordinals, worms also stand in an intimate relation with Turing progressions \cite{Joosten2016}. All of these mathematical entities can be manipulated and reasoned about within the rather simple modal logic $\glp^0$.

Even though the logic $\glp^0$ is already a substantial simplification with respect to $\glp$, its decidability problem is still ${\sf PSPACE}$-complete \cite{Pakhomov2014}.
Furthermore, the problem of frame incompleteness is still there, but the logic $\glp^0$ does have a rather well behaved universal model \cite{BeklemishevJJJVervoort2005}.

A next step in simplifying $\glp^0$ arose by studying strictly positive fragments of $\glp$ and $\glp^0$ by means of the so called \emph{reflection calculi} $\RC$ and $\RCzero$ \cite{Dashkov2012}, \cite{Beklemishev2012}, \cite{Beklemishev2014}. The theorems of $\RC$ and $\RCzero$ are of the form $\varphi \vdash \psi$, where the only connectives in $\varphi$ and $\psi$ are conjunctions and consistency modalities. $\glp$ is conservative over $\RC$, in the sense that for $\varphi$ and $\psi$ only using conjunctions and consistency operators, we have that $\varphi \vdash \psi$ is provable in $\RC$ if and only if $\varphi \to \psi$ is a theorem of $\glp$ \cite{Dashkov2012}.

The reflection calculi are known to be very well-behaved. In particular, the problem of frame-incompleteness is no longer there, and the decision problem is decidable in polynomial time \cite{Dashkov2012}. Yet, as far as applications to ordinal analysis are concerned, no essential expressive power has been lost. Thus, the second step in our simplification brings us from $\glp^0$ to $\RCzero$.

Given the limited signature of $\RCzero$, its formulas are just built from diamonds, conjunctions, and top. However, it is provable in $\RCzero$ that each of its formulas is equivalent to a single worm \cite{Ignatiev1993}. As such, one may wonder if some decent axiomatization of the worm fragment of $\RCzero$ exists that only uses worms and only proves statements of the form $A\vdash B$ with $A$ and $B$ being worms. The current paper settles this question in the positive, presenting a calculus $\WC$ that only manipulates worms, so that $\RC^0$, and thus also $\glp^0$, are conservative extensions of $\WC$.

In the last two sections of the paper we dwell on semantics for $\WC$. In particular we see that although $\WC$ has the finite model property, any (moderately nice) universal model for $\WC$ inherits much of the intrinsic complexity of Ignatiev's universal model for $\glp^0$.

\section{The Reflection Calculus}

Given an ordinal $\Lambda$, the \emph{Reflection Calculus} for $\Lambda$ --- we write $\RC_\Lambda$ --- is a propositional sequent logic in a modal language that is strictly positive. The language is hence composed of $\top$, variables, and closed both under the binary connective $\land$, and the unary modal operators $\la \alpha \ra$ for each ordinal $\alpha < \Lambda$.

\begin{definition}[Reflection Calculus, $\RC_\Lambda$, \cite{Beklemishev2012}]
Let $\varphi, \psi$ and $\chi$ be formulas in the language of $\RC_\Lambda$, and $\alpha, \beta < \Lambda$ be ordinals.

The axioms of $\RC_\Lambda$ are:
\begin{enumerate}[label = \arabic*.]
\item
$\varphi \vdashr \varphi$ and $\varphi \vdashr \top$;

\item
$\varphi \wedge \psi \vdashr \varphi$ and $\varphi \wedge \psi \vdashr \psi$;

\item
$\la \alpha \ra \la \alpha \ra \varphi \vdashr  \la \alpha \ra \varphi$;

\item
$ \la \alpha \ra \varphi \vdashr  \la \beta \ra \varphi$ for $\alpha > \beta$;

\myitem[5] $ \la \alpha \ra \varphi \wedge \la \beta \ra \psi \vdashr  \la \alpha \ra \big ( \varphi \wedge \la \beta \ra \psi \big )$ for $\alpha > \beta$. \label{RC_Axiom5}
\end{enumerate}
The rules are:
\begin{enumerate}[label = \arabic*.]
\item
If $\varphi \vdashr \psi$ and $\psi \vdashr \chi$, then $\varphi \vdashr \chi$;

\item
If $\varphi \vdashr \psi$ and $\varphi \vdashr \chi$, then $\varphi \vdashr \psi \wedge \chi$;

\item
If $\varphi \vdashr \psi$, then $\la \alpha \ra \varphi \vdashr \la \alpha \ra \psi$.
\end{enumerate}

If $\varphi  \vdashr \psi$, we say that $\psi$ follows from $\varphi$ in $\RC$. If both $\varphi \vdashr \psi$ and $\psi \vdashr \varphi$, we say that $\varphi$ and $\psi$ are equivalent in $\RC$, and write $\varphi \equiv_\RC \psi$.
\end{definition}

In this paper we are mainly interested in the closed fragment of $\RC_\Lambda$, denoted by $\RC^0_\Lambda$, which is the same as $\RC_\Lambda$ without variables in the language. Since the following results hold for any chosen $\Lambda$, we will omit it.

There are some inhabitants of $\RCzero$ on which we take special interest: the worms. These are just the formulas of $\RCzero$ that have no $\land$.

\begin{definition}[Worms, $\Worms$ and $\Worms_\alpha$]
	Worms are inductively defined as follows: $\top$ is in $\Worms$; if $A$ is in $\Worms$ and $\alpha$ is an ordinal, then $\la \alpha \ra A$ is in $\Worms$.
	
	Worms whose modalities are all at least $\alpha$ --- we write $\Worms_\alpha$ --- are defined inductively in a similar manner: $\top$ is in $\Worms_\alpha$; if $A$ is in $\Worms_\alpha$ and $\gamma \geq \alpha$ is an ordinal, then $\la \gamma \ra A$ is in $\Worms_\alpha$.
\end{definition}

It is a known result \cite{Ignatiev1993} that any formula in the language of $\RCzero$ is equivalent to a worm.

\begin{lemma}\label{theorem:eachClosedFormulaIsAworm}
For each formula $\varphi$ of $\RCzero$ there is a worm $A$ such that
	$
		\varphi \equiv_\RC A.
	$
\end{lemma}

This makes one wonder whether it would be possible to work with a calculus that only involves worms as far as $\RCzero$ is concerned. This paper settles the question in the positive.

\section{The Worm Calculus}

We propose a \emph{Worm Calculus} --- we write $\WC$ --- which derives sequents of worms.
Since the language of $\WC$ only includes $\top$ and diamonds $\la \alpha \ra$ for an ordinal $\alpha$, we omit the $\la \cdot \ra$, obtaining formulas which are simply strings of ordinals ending in $\top$. To further simplify, for the worms $A\top$ and $B\top$, we will write $A$ and $B$. When we write $AB$ this is understood as $AB\top$.

\begin{definition}[Worm Calculus, $\WC$]
Let $A, B$ and $C$ be worms, and $\alpha, \beta$ be ordinals.

The axioms of $\WC$ are:
\begin{enumerate}[label = A\arabic*.]
\item $A \vdashw \top$;

\myitem[Transitivity] $\alpha \alpha A \vdashw \alpha A$ (Transitivity); \label{Transitivity}

\myitem[Monotonicity] $\alpha A \vdashw \beta A$ for $\alpha > \beta$ (Monotonicity). \label{Monotonicity}
\end{enumerate}

The rules of $\WC$ are:
\begin{enumerate}[label = R\arabic*.]
\item If $A \vdashw B$ and $B \vdashw C$, then $A \vdashw C$ (Cut);

\myitem[Necessitation] If $A \vdashw B$, then $\alpha A \vdashw \alpha B$ (Necessitation); \label{Necessitation}

\myitem[R3] If $A \vdashw B$ and $A \vdashw \alpha C$, then $A \vdashw B \alpha C$, for $B \in \Worms_{\alpha+1}$. \label{Rule3} 
\end{enumerate}

If $A \vdashw B$, we say that $B$ follows from $A$ in $\WC$. If both $A \vdashw B$ and $B \vdashw A$, we say that $A$ and $B$ are equivalent in $\WC$, and write $A \equiv_\WC B$.
\end{definition}

To express recursion and the notion of simplicity, we use a simple measure on worms, their length. The length of a worm is the total number of symbols other than $\top$.

\begin{definition}[Length]
The \emph{length} of a worm $A$ --- we write $\len{A}$ --- is defined recursively as such: $\len{\top} := 0$, and $\len{\alpha A} := \len{A} + 1$.
\end{definition}

We can immediately prove some facts about worms using the worm calculus.

\begin{lemma}\label{lemma:WormImpliesFirstPartOfIt}
For any worms $A$ and $B$, we have that $AB \vdashw A$.
\end{lemma}

\begin{proof}
The proof goes by induction on the length of $A$. Starting from $B\vdash \top$ (base case), repeatedly apply \ref{Necessitation} to build $A$ up front.
\end{proof}

From this lemma we obtain a simple but useful corollary.

\begin{corollary}\label{lemma:ApA}
For any worm $A$, we have that $A\vdashw A$.
\end{corollary}

It is in general not true that $A B \vdashw B$, but there is a special case.

\begin{lemma}\label{lemma:WormImplesSecondPartOfItSometimes}
	For any ordinal $\alpha$ and worms $A$ and $B$ such that $A \in \Worms_{\alpha + 1}$, we have that
	$
		A \alpha B \vdashw \alpha B
	$.
\end{lemma}
\begin{proof}
	By induction on the length of $A$, with the help of \ref{Transitivity} and \ref{Monotonicity}.
\end{proof}

\begin{lemma}\label{lemma:topIsMinimal}
	For any non-trivial worm $A \in \Worms_\alpha$, we have that $A \vdashw \alpha$.
\end{lemma}
\begin{proof}
By induction on the length of $A$. If $\len{A} = 1$, then $A = \beta$ for some $\beta \geq \alpha$. The result follows by \ref{Monotonicity} and Corollary \ref{lemma:ApA}. For the induction step, consider $A = \beta A'$, where $\beta \geq \alpha$ and we already know $A' \vdashw \alpha$. Then by \ref{Necessitation} and \ref{Transitivity}, $\alpha A' \vdashw \alpha \alpha \vdashw \alpha$. Since $\beta A' \vdashw \alpha A'$, we are done.
\end{proof}

It is easy to see that $\RC$ extends $\WC$. As we shall later see, $\RC$ is conservative over $\WC$, which means that this extension is, in a sense, not proper. The first of these two claims is articulated in the following theorem.

\begin{theorem}\label{thm:WCimpliesRC}
For any two worms $A$ and $B$ we have that
	$
		A \vdashw B
	$
	implies
	$
		A \vdashr B
	$.
\end{theorem}

\begin{proof}
By an easy induction on the length of a $\WC$ proof. To see that Rule \ref{Rule3} is admissible in $\RC$, we use induction on the length of $B$ and Axiom \ref{RC_Axiom5}.
\end{proof}

The proof of the converse is a bit more involved. We shall use the fact that an implication between worms can be recursively broken down into implications between simpler worms.

\section{Decomposing worms}

The notions of $\alpha$-head and $\alpha$-remainder are useful to break down worms into smaller ones.

\begin{definition}[$\alpha$-head, $\alpha$-remainder]
	Let $A$ be a worm and $\alpha$ be an ordinal.
	
	The \emph{$\alpha$-head} of $A$ --- we write $h_\alpha(A)$ --- is defined recursively as: $h_\alpha(\top) := \top$, $h_\alpha(\beta A) := \beta h_\alpha(A)$ if $\beta \geq \alpha$, and $h_\alpha(\beta A) := \top$ if $\beta < \alpha$.
	
	Likewise, the \emph{$\alpha$-remainder} of $A$ --- we write $r_\alpha(A)$ --- is defined recursively as: $r_\alpha(\top) := \top$, $r_\alpha(\beta A) := r_\alpha(A)$ if $\beta \geq \alpha$, and $r_\alpha(\beta A) := \beta A$ if $\beta < \alpha$.
\end{definition}

Intuitively, the $\alpha$-head of $A$ is the greatest initial segment of $A$ which is in $\Worms_{\alpha}$, and the $\alpha$-remainder is what remains after cutting off the $\alpha$-head.
It then follows that $A = h_\alpha(A) r_\alpha(A)$, for every worm $A$ and ordinal $\alpha$. An immediate consequence is that the lengths of the $\alpha$-head and of the $\alpha$-remainder of a worm are always at most the length of the worm itself.

It is possible to prove that $A \equiv_\RC h_\alpha(A) \land r_\alpha(A)$ for every worm $A$ and ordinal $\alpha$. In $\WC$ we cannot state such a result due to the lack of the conjunction connective in the language. We can, however, obtain the same consequences.
\begin{lemma}\label{lemma:HeadAndRemainder}
	Let $A$ be a worm and $\alpha$ be an ordinal. Then:
	\begin{enumerate}[ref=(\roman*)]
		\item $A \vdashw h_\alpha(A)$;	\label{item:AndIntroL}
		\item $A \vdashw r_\alpha(A)$;	\label{item:AndIntroR}
		\item If $B \vdashw h_\alpha(A)$ and $B \vdashw r_\alpha(A)$, then $B \vdashw A$.	\label{item:AndExcl}
	\end{enumerate}
\end{lemma}
\begin{proof}
	Note that $A = h_\alpha(A) r_\alpha(A)$, this is to say, they are syntactically the same.
	Thus, Part \ref{item:AndIntroL} follows from Lemma \ref{lemma:WormImpliesFirstPartOfIt}.
	Part \ref{item:AndIntroR} is a consequence of Lemma \ref{lemma:WormImplesSecondPartOfItSometimes}, taking into consideration that $h_\alpha(A) \in \Worms_{\alpha}$ and that $r_\alpha(A)$ always starts with either $\top$ --- making the result trivial --- or with an ordinal less than $\alpha$.
	Part \ref{item:AndExcl} follows from rule \ref{Rule3} unless $r_\alpha(A) = \top$, in which case it is trivial.
\end{proof}

There is another relevant part of a worm, the $\alpha$-body. It is obtained from the $(\alpha + 1)$-remainder by dropping its leftmost modality (as long as said remainder is not trivial).

\begin{definition}[$\alpha$-body]
	Let $A$ be a worm and $\alpha$ an ordinal.
	The \emph{$\alpha$-body} of $A$ --- we write $b_\alpha(A)$ --- is defined from $r_{\alpha + 1}(A)$ as follows: if $r_{\alpha + 1}(A) = \top$ then $b_\alpha(A) := \top$, and if $r_{\alpha + 1}(A) = \beta B$ then $b_\alpha(A) := B$.
\end{definition}

The $\alpha$-body of a non-trivial worm $A$ is particularly useful because its length is always strictly smaller than the length of $A$.
We can also prove a counterpart of Lemma \ref{lemma:HeadAndRemainder} about the $\alpha$-body.

\begin{lemma}\label{lemma:body}
	Let $\alpha$ be an ordinal and $A$ be a non-trivial worm in $\Worms_\alpha$. Then:
	\begin{enumerate}[ref=(\roman*)]
		\item \label{item:body-bALTNA} $A \vdashw \alpha b_\alpha(A)$; \label{item:AndIntroR_body}
		\item If $B \vdashw h_{\alpha + 1}(A)$ and $B \vdashw \alpha b_\alpha(A)$, then $B \vdashw h_{\alpha + 1}(A) \alpha b_{\alpha}(A)$; \label{item:AndExcl1_body}
		\item $h_{\alpha + 1}(A) \alpha b_\alpha(A) \vdashw A$; \label{item:AndExcl2_body}
		\item $A \equiv_\WC h_{\alpha + 1}(A) \alpha b_\alpha(A)$. \label{item:body_summary}
	\end{enumerate}
\end{lemma}
\begin{proof}
	We make a case distinction on $r_{\alpha + 1}(A)$ in order to prove Parts \ref{item:AndIntroR_body} to \ref{item:AndExcl2_body} separately in each case.
	
	Suppose that $r_{\alpha + 1}(A) = \beta b_\alpha(A)$ for some ordinal $\beta$. Since $A \in \Worms_\alpha$, then $\beta \geq \alpha$. But since it is in the $(\alpha + 1)$-remainder, $\beta < \alpha + 1$. We conclude that $\beta = \alpha$, and hence that $r_{\alpha + 1}(A) = \alpha b_\alpha(A)$. Then Parts \ref{item:AndIntroR_body} to \ref{item:AndExcl2_body} are just a corollary of Lemma \ref{lemma:HeadAndRemainder}.

	However it can be the case that $r_{\alpha + 1}(A) = \top$ and hence $b_\alpha(A) = \top$ as well.
	Then Part \ref{item:AndIntroR_body} becomes an instance of Lemma \ref{lemma:topIsMinimal},
	Part \ref{item:AndExcl1_body} follows from Rule \ref{Rule3}
	and Part \ref{item:AndExcl2_body} is a consequence of Lemma \ref{lemma:WormImpliesFirstPartOfIt}.
	
	Finally, Part \ref{item:body_summary} is a corollary of all of the other parts put together.
\end{proof}

The following result is Lemma 3.15 of \cite{FernandezDuqueJJJ2014}. It describes part of a recursive decision procedure for provability in $\RC$ between worms.

\begin{lemma}\label{lemma:RecursiveWormImplication}
For any two worms $A$ and $B$ and for any ordinal $\alpha$ we have that
	$
		A \vdash_\RC \alpha B
	$
	if and only if both
	$
		h_\alpha(A) \vdashr \alpha h_\alpha(B)$ and $A \vdashr r_\alpha(B)
	$.
\end{lemma}

Let us see that we can prove one of the implications in $\WC$, which we will later use in the proof of our main theorem (Theorem \ref{thm:RCimpliesWC}). There is no \emph{a priori} reason why the other implication can't also hold; in fact, we will see that it does, since the calculi are equivalent for worms. It just so happens that we have no use for it.

\begin{lemma}\label{lemma:RecursiveWormImplicationWC}
For any two worms $A$ and $B$, and for any ordinal $\alpha$, if we have that
	$
		h_\alpha(A) \vdashw \alpha h_\alpha(B)$ and $A \vdashw r_\alpha(B)
	$,
	then we have $A \vdashw \alpha B$.
\end{lemma}

\begin{proof}
	Taking into consideration that $A = h_\alpha(A) r_\alpha(A)$ and similarly for $B$, consider two cases. In the first case, $r_\alpha(B) = \top$, and this is a consequence of Lemma \ref{lemma:HeadAndRemainder}. In the second case, $r_\alpha(B) = \beta C$ for some $\beta < \alpha$ and worm $C$. Then the result follows from Rule \ref{Rule3}.
\end{proof}

We now want to prove that $\RC$ is conservative over $\WC$ using the following inductive strategy. If $A \vdashr B$, we use Lemma \ref{lemma:RecursiveWormImplication} to recast this into a collection of provability statements in $\RC$ between worms with smaller lengths. We then translate them to $\WC$ using the induction hypothesis, and finally go back with the help of Lemma \ref{lemma:RecursiveWormImplicationWC}. However, depending on the worms $A$ and $B$, it could be the case that these two theorems are not enough, since they don't always reduce the length of the provability statements. In what follows, we introduce some more useful notions and results, which will help us deal with that problem.

\section{Well founded orders on worms}

It is possible to define an order relation between worms as is standard in the literature.
\begin{definition}[Ordering worms]
	We say that $A <_\alpha B$ if $B \vdashw \alpha A$. Furthermore, we say that $A \leq_\alpha B$ if either $A <_\alpha$ or $A \equiv_{\WC} B$. The provability can be taken in $\RC$ to obtain $<^{\RC}_\alpha$ and $\leq^{\RC}_\alpha$, respectively.
\end{definition}

It is well-known that $<^{\RC}_\alpha$ is irreflexive \cite{BeklemishevFernandezDuqueJJJ2014}. Since $\WC$ is embedded in $\RC$, we also know that $<_\alpha$ is irreflexive.
It is easy to see that both relations are transitive.

Our goal now is to show that $<_\alpha$ is a total relation over worms in $\Worms_\alpha$. This has been shown for $<^\RC_\alpha$ using worm normal forms \cite{BeklemishevFernandezDuqueJJJ2014}, but here we follow a different strategy, proposed in \cite{FernandezDuque2017}. We start by presenting a number of useful sufficient conditions to deduce $A <_\alpha B$, and one to deduce $A \equiv B$.

\begin{lemma}\label{lemma:recursionBody}
	Let $A, B \in \Worms_\alpha$ such that $A, B \neq \top$. Then in $\WC$ (and hence in $\RC$) we have the following:
	\begin{enumerate}
		\item If $b_\alpha(B) \vdash A$, then $A <_\alpha B$;
		\item If $A <_\alpha b_\alpha(B)$, then $A <_\alpha B$;
		\item If $b_\alpha(A) <_\alpha B$ and $h_{\alpha + 1}(A) <_{\alpha + 1} h_{\alpha + 1}(B)$, then $A <_{\alpha + 1} B$ (and consequently $A <_{\alpha} B$);
		\item If $b_\alpha(A) <_\alpha B$ and $b_\alpha(B) <_\alpha A$ and $h_{\alpha + 1}(A) \equiv h_{\alpha + 1}(B)$, then $A \equiv B$.
	\end{enumerate}
\end{lemma}
\begin{proof} 
	For the first item, from $b_\alpha(B) \vdash A$ we get by \ref{Necessitation} that $\alpha b_\alpha(B) \vdash \alpha A$. Since by Lemma \ref{lemma:body} we know that $B \vdash \alpha b_\alpha(B)$, we can conclude that $B \vdash \alpha A$.
  The second item follows by the transitivity of $<_\alpha$, taking into account that $b_\alpha(B) <_\alpha B$ (Lemma \ref{lemma:body}.\ref{item:body-bALTNA}).
	The third item follows from Lemma \ref{lemma:RecursiveWormImplicationWC}.	
	Finally, for the fourth item we use $B \vdash \alpha b_\alpha(A)$ and $h_{\alpha + 1}(B) \vdash h_{\alpha + 1}(A)$ to get $B \vdash h_{\alpha + 1}(A) \alpha b_\alpha(A)$, and hence $B \vdash A$. Then we obtain $A \vdash B$ in the same way.
\end{proof}

Now we are ready to prove the totality of $<_\alpha$ for worms in $\Worms_\alpha$.

\begin{lemma}[Trichotomy, \cite{FernandezDuque2017}]\label{lemma:trichotomy}
	Given worms $A, B \in \Worms_\alpha$, we have that either $A <_\alpha B$, or $A \equiv B$, or $B <_\alpha A$.
\end{lemma}
\begin{proof}
  We proceed by induction on the length of $A B$. If the length is zero, \emph{i.e.}, if $A = B = \top$, then clearly $A \equiv B$.
        
  Note that by Lemma \ref{lemma:topIsMinimal}, $\top <_\alpha C$ regardless of the worm $C \in \Worms_\alpha$, as long as $C \neq \top$. Then if exactly one of $A, B$ is $\top$ we have also solved our problem.

  Now for the induction step, take both $A$ and $B$ with positive length. Our induction hypothesis is:
  \begin{quote}
    For any ordinal $\beta$ and worms $C, D \in \Worms_\beta$ such that $\len{C D} < \len{A B}$, we have $C <_\beta D$, or $C \equiv D$, or $D <_\beta C$.
  \end{quote}

  Let $\xi$ be the minimum ordinal in $A B$, which means that $\alpha \leq \xi$. According to Lemma \ref{lemma:recursionBody}, if $A \leq_\xi b_\xi(B)$ or $B \leq_\xi b_\xi(A)$, we can conclude $A <_\xi B$ or $B <_\xi A$, respectively. Assume then that $A \not\leq_\xi b_\xi(B)$ and $B \not\leq_\xi b_\xi(A)$. Since we took $A \neq \top$, it is clear that $\len{b_\xi(A)} < \len{A}$, and analogously for $B$. Then by the induction hypothesis, we have $b_\xi(B) <_\xi A$ and $b_\xi(A) <_\xi B$.
                        
  Since we are assuming $\xi$ is in $A B$, we also know that
    \begin{equation*}
      \len{h_{\xi + 1}(A) h_{\xi + 1}(B)} < \len{A B}
      ,
    \end{equation*}
  and thus by the induction hypothesis we have
    \begin{gather*}
      h_{\xi + 1}(A) <_{\xi + 1} h_{\xi + 1}(B), \text{ or} \\
      h_{\xi + 1}(B) <_{\xi + 1} h_{\xi + 1}(A), \text{ or} \\
      h_{\xi + 1}(A) \equiv h_{\xi + 1}(B).
    \end{gather*}
  In the first two cases we use Lemma \ref{lemma:recursionBody} again to conclude $A <_\xi B$ or $B <_\xi A$ respectively. In the last case we can use the same lemma to get $A \equiv B$.

  As a final remark, we observe that since $\alpha \leq \xi$, we have that $C <_\xi D$ implies $C <_\alpha D$ for any worms $C, D \in \Worms_\xi$.  
\end{proof}

With the totality of $<_\alpha$ at hand, we can already show that proofs in $\RC$ and $\WC$ are equivalent for certain worms.

\begin{theorem}
\label{theorem:A_ltn_B_iff_A_ltnRC_B}
  If $A, B \in \Worms_\alpha$, then:
  \begin{align*}
    A <_\alpha B &\iff A <_\alpha^\RC B; \\
    A \equiv_\WC B &\iff A \equiv_\RC B.
  \end{align*}
\end{theorem}
\begin{proof}
  Both left-to-right implications are a consequence of $\RC$ extending $\WC$ (Theorem \ref{thm:WCimpliesRC}).

  For the right-to-left implication of the first statement, we reason as follows: suppose that $A <_\alpha^\RC B$ but it is not the case that $A <_\alpha B$. Then by the totality of $<_\alpha$ for worms in $\Worms_\alpha$ (Lemma \ref{lemma:trichotomy}), either $B <_\alpha A$, or $A \equiv_\WC B$. By the inclusion of $\WC$ in $\RC$, we then conclude that either $B <_\alpha^\RC A$, or $A \equiv_\RC B$. But neither of these two cases is possible, as they contradict the irreflexivity of $<_\alpha^\RC$. Then it must be the case that $A <_\alpha B$.
  The proof of the second statement is analogous.
\end{proof}

\section{Conservativity of \texorpdfstring{$\RC$}{RC} over \texorpdfstring{$\WC$}{WC}}

We are now ready to prove the main result of this paper.

\begin{theorem}\label{thm:RCimpliesWC}
For any worms $A$ and $B'$, if 
	$
		A\vdashr B',
	$
	then also
	$
		A\vdashw B'.
	$
\end{theorem}

\begin{proof}
	If $B' = \top$, the result is immediate. Assume then that $B' = \alpha B$ for some ordinal $\alpha$ and worm $B$.
	
	The proof proceeds by complete induction on the length of $A \alpha B$. The minimum length of $A \alpha B$ is 1, and it occurs only when $A = B = \top$. But the premise $\top \vdashr \alpha$ is absurd since it would contradict the irreflexivity of $<^\RC_\alpha$, and hence there is nothing left to prove.
	
	For the induction step, note that our induction hypothesis is the following:
  \begin{quote}
    For any worms $C, D$ such that $\len{C D} < \len{A \alpha B}$ and $C \vdashr D$, we have that $C \vdashw D$.
  \end{quote}
	Assume that $A \vdashr \alpha B$. From Lemma \ref{lemma:RecursiveWormImplication} we get $h_\alpha(A) \vdashr \alpha h_\alpha(B)$ and $A \vdashr r_\alpha(B)$.
	Consider the following cases:
	\begin{enumerate}
		\item $r_\alpha(A) \neq \top$ or $r_\alpha(B) \neq \top$.
		
			Then, since $\len{A} = \len{h_\alpha(A)} + \len{r_\alpha(A)}$ (and equivalently for $B$), we know that $\len{h_\alpha(A)} < \len{A}$ or $\len{h_\alpha(B)} < \len{B}$. Then
				$\len{h_\alpha(A) \alpha h_\alpha(B)} < \len{A \alpha B}.$
			Furthermore, $\len{A r_\alpha(B)} < \len{A \alpha B}$. By using the induction hypothesis twice we get $h_\alpha(A) \vdashw \alpha h_\alpha(B)$ and $A \vdashw r_\alpha(B)$, which is enough to show $A \vdashw \alpha B$ by Lemma \ref{lemma:RecursiveWormImplicationWC}.
			
		\item $r_\alpha(A) = \top$ and $r_\alpha(B) = \top$.
		
      In this case, we know that $A, B \in \Worms_\alpha$, and hence that $A \vdash_\WC \alpha B$ by Theorem \ref{theorem:A_ltn_B_iff_A_ltnRC_B}.
	\end{enumerate}
\end{proof}

The proof of the preceding result (together with the proofs of the results it uses) gives us a constructive algorithm to decide whether $A \vdashw B$. Furthermore, if indeed $A \vdashw B$, this algorithm provides a list of syntactical steps which form a formal proof. Since at each iteration of the recursion we may need to decide several different statements with only slightly smaller lengths, the algorithm is exponential. It is known that there is a polynomial procedure to decide $\RC$ \cite{Dashkov2012} (and hence, as we've seen, $\WC$), but it uses semantics. Finding a polynomial syntactical algorithm remains an open problem.

Combining Theorems \ref{thm:WCimpliesRC} and \ref{thm:RCimpliesWC}, we obtain the promised result: $\RC$ is a conservative extension of $\WC$.

\begin{theorem}
For any worms $A$ and $B$ we have that
	$
		A\vdash_\RC B
	$
	if and only if
	$
		A\vdash_\WC B
	$.
\end{theorem}

Combining this theorem with Lemma \ref{theorem:eachClosedFormulaIsAworm} we obtain the following corollary.

\begin{corollary}
Let $\varphi$ and $\psi$ be closed $\RC$-formulas with corresponding worms $A$ and $B$ such that $\varphi \equiv_\RC A$ and $\psi \equiv_\RC B$. Then we have $
		\varphi \vdashr \psi$ if and only if $A\vdashw B$.
\end{corollary}

\section{Relational semantics and Ignatiev's model}

Let us briefly recall how we arrived at the calculus $\WC$ while we comment on the relational semantics for the intermediate steps. Japaridze went from the regular provability logic $\gl$ to its polymodal version $\glp$. Whereas $\gl$ is frame-complete, it turned out that $\glp$ is frame incomplete. Ignatiev intensively studied the closed fragment $\glp^0$ and --- although the frame incompleteness is still salient --- introduced a universal model $\mathcal{I}$ for it. Ignatiev's model $\mathcal{I}$ is essentially infinite, having fractal features.

Dashkov and Beklemishev studied reflection calculi and in particular the strictly positive fragments $\RC$ and $\RCzero$ of $\glp$ and $\glp^0$, respectively. Here the only connectives are the diamond modalities together with conjunctions. The reflection calculi are known to be frame complete and have the finite model property. Furthermore, linear frames ($xRy$ and $xRz$ imply $y=z$, $yRz$ or $zRy$) suffice for the closed fragment \cite{FernandezDuque2017}.

In this paper we perform a final simplification on $\RC^0$, getting rid of the conjunctions to end up with $\WC$. Inspired by the finite model property of $\RC^0$ and whence of $\WC$, in the last section of this paper we question whether $\WC$ may have a universal model $\mathcal{U}$ that is significantly simpler than Ignatiev's model $\mathcal{I}$. We settle the answer to this question with a yes and a no. 

Yes, $\mathcal U$ can be simpler in that we can bound the length of the strict chains of successors in $\mathcal U$ by $\omega$. For this, it suffices to take the disjoint union of all finite $\RC^0$ counter models for all statements for which $A\nvdash B$. This clearly defines a universal model for $\RC^0$ with only finite strict $R_\alpha$-chains whereas, as we shall see below, the model $\mathcal I$ has arbitrarily long strict $R_\alpha$-chains. On the other hand, we shall see that for a large class of universal models $\mathcal U$, they inherit much of the intrinsic complexity of $\mathcal I$ in that for infinitely many essentially different points $x\in \mathcal I$ we can find corresponding points $y\in \mathcal U$ such that $x$ and $y$ have the same modal theory.

Before we can make this statement precise, we need a couple of technical definitions that allow us to describe Ignatiev's model $\mathcal I$. As a first step, we need to define the \emph{end-logarithm} $\ell$ as a function from the ordinals to the ordinals by stipulating $\ell(0):=0$ and $\ell(\alpha + \omega^\beta) := \beta$. 
Next, we need to define iterates of $\ell$  --- the \emph{hyper-logarithms} --- and write $\ell^\xi$ to denote the \emph{$\xi$-th iterate of $\ell$}. We define: 
\begin{quote}
\begin{enumerate}
\item
$\ell^0 := {\sf id}$, 

\item
$\ell^1 := \ell$ and,

\item
$\ell^{\alpha + \beta} : = \ell^\beta \circ \ell^\alpha$.
\end{enumerate} 
\end{quote}

Clearly, these three properties do not tell us anything about $\ell^\xi$ for an additively indecomposable $\xi$. To fix that, we will use the notion of \emph{initial function}. An initial function on the ordinals is a function that maps initial segments $[0, \ldots, \alpha]$ onto initial segments $[0, \ldots, \beta]$. 

We now further require that each $\ell^\xi$ is point-wise maximal among all families  of initial ordinal functions $\{ f^\xi\}_{\xi \in {\sf On}}$ that satisfy the three properties. In this way, clearly each $f^\xi$ defines an initial function. For the purposes of this paper, many of the exact details of the $\ell^\xi$ functions are irrelevant and we refer the interested reader to \cite{FernandezDuqueJJJ2013} or \cite{FernandezDuqueJJJ2013_hyperations} for further details.

Let us, by way of example, compute the first couple of values of $\ell^\omega$. Recall that as always, $\varepsilon_\zeta$ denotes the $\zeta$-th fixpoint of $x\mapsto \omega^x$. Since the initial segment $[0]$ should be mapped onto an initial segment, it must be that $\ell^\omega(0) =0$. If $\alpha< \varepsilon_0$ then it is easy to see that for some $n<\omega$ we have that $\ell^n(\alpha) = 0$. Consequently, $\ell^\omega(\alpha) = \ell^{n+\omega}(\alpha) = \ell^{\omega} \circ \ell^{n}(\alpha) = \ell^{\omega} (0) = 0$.
Consequently, each initial segment $[0,\ldots, \alpha]$ for $\alpha<\varepsilon_0$ will be mapped by $\ell^\omega$ to the initial segment $[0]$.

What about $\ell^\omega (\varepsilon_0)$? If we disregard the requirement on initiality, it is not hard to see that $\ell^\omega (\varepsilon_0)$ could be \emph{any} value. However, since $\ell^\omega$ should map the initial segment $[0, \ldots, \varepsilon_0]$ to an initial segment, the maximal possible value doing so requires that $\ell^\omega(\varepsilon_0) = 1$. 

We observe that $\ell^\omega(\xi + \zeta) = \ell^{1+\omega}(\xi + \zeta)= \ell^\omega\circ \ell(\xi + \zeta) =\ell^\omega\circ \ell(\zeta) = \ell^{1+\omega}(\zeta) = \ell^\omega(\zeta)$ so that $\ell^\omega (\varepsilon_0 + \alpha) = 0$ for any $\alpha<\varepsilon_0$. 

Following this kind of arguments, we can now see that the next value where $\ell^\omega$ increases will be at $\varepsilon_1$ and $\ell^\omega (\varepsilon_1) = 2$. 
Fortunately, we do not need to prove tons of theorems any time we are required to know some value of $\ell^\alpha(\beta)$ and in \cite{FernandezDuqueJJJ2013_hyperations,FernandezDuqueJJJ2013} a recursive algorithm is presented to compute these values:

\begin{proposition}
For ordinals $\xi,\zeta$, the following recursion is well-defined and determines all the $\ell^\xi(\zeta)$ values:
\begin{enumerate}
  \item $\ell^0(\alpha)=\alpha$,
  \item $\ell^\xi(0)=0$,
  \item $\ell^1(\alpha+\omega^\beta)=\beta$,
  \item $\ell^{\omega^\rho+\xi}=\ell^{\xi} \circ \ell^{\omega^\rho}$ provided that $\xi<\omega^\rho+\xi$,
  \item $\ell^{\omega^\rho}(\zeta) = \ell^{\omega^\rho}(\ell^\eta(\zeta))$ if $\rho>0$ and $\eta<\omega^\rho$ is such that $\ell^\eta(\zeta)<\zeta$,
  \item $\ell^{\omega^\rho}(\zeta)=\displaystyle\sup_{\delta\in [0,\zeta)}(\ell^{\omega^\rho}(\delta)+1)$ if $\rho>0$ and $\zeta=\ell^\eta(\zeta)$ for all $\eta<\omega^\rho$.
\end{enumerate}
\end{proposition}

Now that the hyper-logarithms have been defined, we can specify the points of Ignatiev's model $\mathcal I$ which are the so-called \emph{$\ell$-sequences}.

\begin{definition}[$\ell$-sequence]
An \emph{$\ell$-sequence} is a function $f : {\sf On} \to {\sf On}$ such that for each ordinal $\zeta$ we have $f(\zeta) \leq \ell^{-\xi + \zeta}\big(f (\xi)\big)$ for $\xi<\zeta$ large enough.
\end{definition}

At times we shall write $f_\xi$ instead of $f(\xi)$. We note that for each $\ell$-sequence $f$ the inequality $f(\alpha +1)\leq \ell \big ( f(\alpha)\big)$ holds. Furthermore, the requirement of $\xi<\zeta$ being large enough is important, as it means that
	$
		f = \la \omega^{\varepsilon_0+1},\varepsilon_0, \varepsilon_0, \ldots 1, 0 \ldots \ra
	$
	is an $\ell$-sequence, where $f(0) = \omega^{\varepsilon_0 +1}$, $f(i) = \varepsilon_0$ for $0<i<\omega$, $f(\omega)=1$ and $f(i) = 0$ for $i>\omega$. It is easy to see that $\ell^\omega(\omega^{\varepsilon_0 + 1}) = 0$ and $\ell^\omega(\varepsilon_0) = 1$. Then $f(\omega) \leq \ell^\omega \big(f(1)\big)$ but it is not the case that $f(\omega) \leq \ell^\omega \big(f(0)\big)$.

We can now define the class-size version of Ignatiev's model as the collection of all $\ell$-sequences with suitable relations $R_\xi$ to model each of the $\la \xi \ra$ modalities. For all practical purposes we can take sufficiently large set-size truncations of the class-size model.

\begin{definition}[Ignatiev's model]
Ignatiev's model is $\mathcal I := \la I, \{R_\xi \}_{\xi \in {\sf On}}\ra$, where $I$ is the collection of all $\ell$-sequences and $f R_\xi g$ if both $f(\alpha) = g(\alpha)$ for $\alpha<\xi$ and $f(\xi) > g(\xi)$.
\end{definition}

For example, we can see that
\[
\begin{array}{llll}
 &\la \omega^{\varepsilon_0+1},\varepsilon_0, \varepsilon_0, \ldots 0, 0 \ldots \ra  &R_0 & \la {\varepsilon_0},\varepsilon_0, \varepsilon_0, \ldots 0, 0 \ldots \ra, \\
 &\la \omega^{\varepsilon_0+1},\varepsilon_0, \varepsilon_0, \ldots 1, 0 \ldots \ra  &R_0 & \la {\varepsilon_0},\varepsilon_0, \varepsilon_0, \ldots 1, 0 \ldots \ra, \\
 & \la \omega^{\varepsilon_0+1},\varepsilon_0, \varepsilon_0, \ldots 1, 0 \ldots \ra  &R_\omega & \la \omega^{\varepsilon_0+1},\varepsilon_0, \varepsilon_0, \ldots 0, 0 \ldots \ra, \\
\mbox{but \ \ \ \ \ \ } \neg \Big( &   \la \omega^{\varepsilon_0+1},\varepsilon_0, \varepsilon_0, \ldots 1, 0 \ldots \ra 
  &R_{\omega} & \la {\varepsilon_0},\varepsilon_0, \varepsilon_0, \ldots 0, 0 \ldots \ra \ \ \Big).\\

\end{array}
\]

The relation $\mathcal I, x\Vdash \varphi$ is defined as usual, where we omit the mention of the model $\mathcal I$:
	$x\Vdash \top$;
	$x\Vdash \varphi \wedge \psi :\Leftrightarrow x\Vdash \varphi$ and $x\Vdash \psi$;
	$x\Vdash \neg \varphi :\Leftrightarrow x \not\Vdash \varphi$;
	and finally, $x\Vdash \la \xi\ra \varphi :\Leftrightarrow$ there is a $y$ s.t.\ $xR_\xi y$ and $y \Vdash \varphi$. 

\begin{theorem}[\cite{FernandezDuqueJJJ2013}]
$\glp^0$ is sound and complete with respect to Ignatiev's model, that is
	$
		\glp^0 \vdash \varphi
	$
	if and only if
	$
		\forall x \in \mathcal I \ x \Vdash \varphi.
	$
\end{theorem} 

We now define an important subset of $\mathcal I$ that has rather nice properties.

\begin{definition}[Main axis, $\sf MA$]
  By $\sf MA$ we denote the \emph{main axis} of Ignatiev's model $\mathcal I$ and define it as such: $f\in {\sf MA} :\Leftrightarrow \forall \zeta \  \forall \, \xi {<} \zeta \ f(\zeta) = \ell^{-\xi + \zeta}(f(\xi))$.
\end{definition} 

For example, $\la \omega^{\varepsilon_0+1},\varepsilon_0, \varepsilon_0, \ldots 1, 0 \ldots \ra$ is not on the main axis, whereas 
$\la \omega^{\varepsilon_0+1},\varepsilon_0+1, 0 \ldots 0, 0 \ldots \ra$ is.
One of the nice properties of the main axis is that each point on it is modally definable. We refer the reader to \cite{FernandezDuqueJJJ2013} for a proof of the following.

\begin{lemma}\label{lemma:mainAxisElementsAndWorms}
For each $x \in {\sf MA}$ there is a worm $A$ such that $\mathcal I ,y \Vdash A \land [0] \neg A$ if and only if $x=y$.
Moreover, each worm $A$ defines a point at the main axis via $A \wedge [0]\neg A$.
\end{lemma}

\section{On universal models for \texorpdfstring{$\WC$}{WC}}

The current proof of the main result of this section does not hold for any universal model but only for models that satisfy an additional natural condition.
Let us recall that in the context of provability logics, a model is called \emph{Euclidean} whenever $xR_\alpha y$ and $xR_\beta z$ imply $y R_\beta z$ for $\beta < \alpha$. Furthermore, by ${\sf Th}(x)$ we denote the collection of worms $\{ A \mid x\Vdash A \}$. Now we are able to state the main result.

\begin{theorem}
Let $\mathcal U$ be a Euclidean universal model for $\WC$. We have that for each point $x\in \mathcal I$ with $x\in {\sf MA}$, there is some $y\in \mathcal U$ such that ${\sf Th}(x)={\sf Th}(y)$.
\end{theorem}

\begin{proof}
Let $x \in \sf MA$ be arbitrary and let $A$ be the worm given by Lemma \ref{lemma:mainAxisElementsAndWorms} such that $A\wedge [0] \neg A$ is true at $x$ and nowhere else. Since $A \nvdash_{\WC} 0A$, we can find $y\in \mathcal U$ with $\mathcal{U}, y\Vdash A$ and $\mathcal{U}, y\nVdash 0A$. We shall show that for this particular choice of $y$ we have ${\sf Th}(x)={\sf Th}(y)$.

First, we assume that $\mathcal{I}, x\Vdash B$ for some worm $B$. By the definability of $x$ and the completeness of $\mathcal I$, we know that $\glp \vdash A\wedge [0] \neg A  \to B$. By Lemma~\ref{lemma:admissibleWormRule}, which we prove below, and the conservativity of $\glp$ over $\WC$, we may conclude that actually $A\vdashw B$, and hence $\mathcal{U}, y\Vdash B$.

Now assume that $\mathcal{I}, x\nVdash B$ for some worm $B$. Then $\glp\nvdash A \to B$, which means that $A \nvdash_{\WC} B$. Let $C$ be a worm equivalent to $A\wedge B$. Clearly, by the trichotomy of $<_0$ and since $A\nvdash_{\WC} B$ and $A \nvdash_{\RC} \la 0 \ra (A \land B)$, we have that $A<_0 C$, whence $C\vdashw 0A$. We assume for a contradiction that $\mathcal{U}, y\Vdash B$. In that case, since also $\mathcal{U}, y\Vdash A$, we may conclude by Lemma \ref{lemma:WCmodelHaveSomeSortOfConjunction} below that $\mathcal{U}, y\Vdash C$. But since $C\vdashw 0A$, this would mean that $\mathcal{U}, y\Vdash 0A$ which is a contradiction by our choice of $y$.
\end{proof}

We finish the paper by proving the two critical lemmas that were needed in the above proof, together with some auxiliary observations. First we define a relation $y\succeq x$ on $\mathcal I$ as $y$ being point-wise at least $x$, that is, $y \succeq x$ if and only if for all $\xi$ we have $y_\xi \geq x_\xi$. The following lemma tells us that this relation, together with the point $x$ where a worm $A$ is true for the first time, characterizes all the points where $A$ holds.

\begin{lemma}
Let $A$ be a worm and $x,y \in \mathcal I$. We have that
	\begin{align*}
		x\Vdash A
		& \implies
		\Big(
			y \succeq x \implies y \Vdash A
		\Big);\\		
		x\Vdash A \land [0] \neg A
		& \implies
		\Big(
			y \succeq x \iff y \Vdash A
		\Big). 		
	\end{align*}
\end{lemma}

\begin{proof}
During this proof we make use of a specific operation on $\ell$-sequences: for an ordinal $\alpha$ and $\ell$-sequences $f$ and $g$, we define $\alpha(f, g)$ to be the ordinal sequence such that $\alpha(f, g)_\zeta = f_\zeta$ for $\zeta<\alpha$ and $\alpha(f, g)_\zeta = g_\zeta$ for $\zeta\geq\alpha$. Clearly, whenever $g_\alpha \leq f_\alpha$, we have that $\alpha(f, g)$ is again an $\ell$-sequence.

The first item is proven by an easy induction on $A$. It was already observed as Lemma 2.4.3.~of \cite{Icard2008}. Note also that the $\Longrightarrow$ direction of the second item follows from the first one.

For the $\Longleftarrow$ direction of the second item, we fix $x$ such that $x\Vdash A \wedge [0] \neg A$, consider $y$ such that $y\Vdash A$, and assume for a contradiction that $y \not\succeq x$. Let $\xi$ be the smallest ordinal such that $y_\xi < x_\xi$. Then it is easy to see that $xR_\xi \xi(x,y)$. 

We will set out to prove the following {\bf claim}: for any worm $B$ and any $\ell$-sequences $f$ and $g$ with $f_\alpha > g_\alpha$ we have that if $f\Vdash B$ and $g\Vdash B$, then $\alpha(f, g)\Vdash B$.

Clearly the result follows from the claim, as it would imply that $\xi(x,y)\Vdash A$, and hence that $x\Vdash \la \xi \ra A$. Consequently, also $x\Vdash \la 0 \ra A$, which contradicts the assumption that $x\Vdash [0] \neg A$.

We prove the claim by induction on $B$ with the base case being trivial. Thus we consider the inductive case where $B = \la \zeta \ra C$ assuming $f\Vdash \la \zeta \ra C$ and $g\Vdash \la \zeta \ra C$.

In the case where $\zeta<\alpha$, we see that for any $w$ such that $fR_\zeta w$, we also have $\alpha(f, g) R_\zeta w$, which tells us that $\alpha(f, g)\Vdash \la \zeta \ra C$.

In the case where $\zeta\geq\alpha$, we find $w$ and $w'$ such that $fR_\zeta w\Vdash C$ and $gR_\zeta w'\Vdash C$. By the induction hypothesis, we know that $\alpha(w,w')\Vdash C$. But since $\alpha(f, g)R_\zeta \alpha(w,w')$, we see that $\alpha(f, g)\Vdash \la \zeta \ra C$ as was to be shown.
\end{proof}

With this characterization lemma, we can easily prove the following admissible rule.

\begin{lemma}\label{lemma:admissibleWormRule}
Let $A$ and $B$ be worms such that $\glp \vdash A \wedge [0] \neg A \to B$. Then we have that $\glp \vdash A \to B$.
\end{lemma}
 
\begin{proof}
Given $A$, let $x$ be the unique $\ell$-sequence where $A \wedge [0] \neg A$ holds. Clearly, since $\glp \vdash A\wedge [0]\neg A \to B$, we also have that $x\Vdash B$. We prove that for any $y$, if $y\Vdash A$, then $y\Vdash B$, from which the result follows by completeness. By the previous lemma, if $y\Vdash A$, then $y\succeq x$. But then, using the previous lemma again, we may conclude that $y\Vdash B$.
\end{proof}

The next two lemmas relate to conjunctions and models of $\WC$.

\begin{lemma}\label{lemma:semanticalBigWormsEatsSmall}
Let $\mathcal{U}$ be a Euclidean universal model for $\WC$ and $x \in \mathcal U$. Then, if $x\Vdash A$ with $A\in \Worms_{\alpha +1}$ and $x\Vdash \alpha  B$, it also holds that $x \Vdash A  \alpha  B$.
\end{lemma}

\begin{proof}
By induction on the length of $A$ with the base case being trivial. For the inductive case, suppose that $x\Vdash \gamma A$ for some $\gamma > \alpha$, and that $x\Vdash \alpha B$.
Then there is $y \in \mathcal{U}$ such that $xR_\gamma y$ and $y\Vdash A$. Likewise, there is $z \in \mathcal{U}$ such that $xR_\alpha z$ and $z\Vdash B$.
Since $\mathcal{U}$ is Euclidean, we know that $y R_\alpha z$, and hence by induction hypothesis that $y \Vdash A \alpha B$. Then clearly $x \Vdash \gamma A \alpha B$.
\end{proof}

With this lemma at hand we can show that although a Euclidean $\WC$-model $\mathcal U$ cannot speak directly about conjunctions, it can indirectly do so.

\begin{lemma}\label{lemma:WCmodelHaveSomeSortOfConjunction}
Let $\mathcal U$ be a Euclidean universal model for $\WC$ and $x\in \mathcal U$. Let $A, B$ and $C$ be worms such that $A \land B \equiv_\RC C$. Then we have $x \Vdash A$ and $x \Vdash B$ if and only if $x\Vdash C$.
\end{lemma}

\begin{proof}
The right-to-left direction is trivial. The left-to-right direction follows by induction on the number of different symbols of $AB$ (the \emph{width} of $AB$) following the standard proof that worms are closed under conjunctions (Lemma 9 of \cite{Beklemishev2005} and Corollary 4.13 of \cite{BeklemishevFernandezDuqueJJJ2014}). The base case is trivial. For the induction step, let $\alpha$ be the minimal modality of $AB$. We know by Lemmas \ref{lemma:HeadAndRemainder} and \ref{lemma:body} that $x\Vdash h_{\alpha+1}(A)$, $x\Vdash \alpha b_\alpha(A)$, $x\Vdash h_{\alpha+1}(B)$, and $x\Vdash \alpha b_\alpha(B)$. Let $D$ be provably equivalent to $h_{\alpha+1}(A) \land h_{\alpha+1}(B)$. Then, since there is no $\alpha$ in the $(\alpha+1)$-heads of $A$ and $B$, we know that $x\Vdash D$ by the induction hypothesis.
By Corollary 4.12 of \cite{BeklemishevFernandezDuqueJJJ2014}, we know that either $\alpha b_\alpha(A) \vdash \alpha b_\alpha(B)$ or $\alpha b_\alpha(B) \vdash \alpha b_\alpha(A)$. Let $\alpha E$ be the maximum. We obtain $x\Vdash D\alpha E$ by Lemma \ref{lemma:semanticalBigWormsEatsSmall}, and clearly $D\alpha E \equiv_\RC A\wedge B$.
\end{proof}

We conclude by observing that all the results proven about universal (Euclidean) models of $\WC$ also hold for universal (Euclidean) models of $\RCzero$. It remains to see whether we can prove the same results for non-Euclidean models.


\end{document}